\documentclass[12pt,leqno]{article}
\usepackage{url}
\usepackage[doc]{optional}
\usepackage{color}
\usepackage{float}
\usepackage{soul}
\usepackage{graphicx}
\definecolor{labelkey}{rgb}{0,0.08,0.45}
\definecolor{refkey}{rgb}{0,0.6,0.0}
\definecolor{Brown}{rgb}{0.45,0.0,0.05}
\definecolor{lime}{rgb}{0.00,0.8,0.0}
\definecolor{lblue}{rgb}{0.5,0.5,0.99}
\usepackage{mathpazo}
\usepackage{amsmath}
\usepackage{amssymb}
\usepackage{theorem}
\usepackage{exscale,relsize}
\usepackage{pifont}
\newcommand{\timess}{\,{\textstyle\mathsmaller{\text{\ding{75}}}}}
\newcommand{\GX}{\ensuremath{\Gamma}}
\newcommand{\aw}{\ensuremath{\stackrel{\mathrm{e}}{\to}}}
\newcommand{\pw}{\ensuremath{\stackrel{\mathrm{p}}{\to}}}

\newcommand{\nnn}{\ensuremath{{n\in{\mathbb N}}}}
\newcommand{\thalb}{\ensuremath{\tfrac{1}{2}}}
\newcommand{\menge}[2]{\big\{{#1}~\big |~{#2}\big\}}
\newcommand{\mmenge}[2]{\bigg\{{#1}~\bigg |~{#2}\bigg\}}

\newcommand{\fenv}[1]%
{\ensuremath{\,\overrightarrow{\operatorname{env}}_{#1}}}
\newcommand{\benv}[1]%
{\ensuremath{\,\overleftarrow{\operatorname{env}}_{#1}}}

\newcommand{\scal}[2]{\left\langle{#1},{#2}  \right\rangle}

\newcommand{\exi}{\ensuremath{\exists\,}}

\newcommand{\RR}{\ensuremath{\mathbb R}}
\newcommand{\RP}{\ensuremath{\mathbb{R}_+}}

\newcommand{\RX}{\ensuremath{\,\left]-\infty,+\infty\right]}}

\newcommand{\NN}{\ensuremath{\mathbb N}}
\newcommand{\CC}{\ensuremath{\mathbb C}}

\newcommand{\dom}{\ensuremath{\operatorname{dom}}}

\newcommand{\ran}{\ensuremath{\operatorname{ran}}}
\newcommand{\conv}{\ensuremath{\operatorname{conv}}}

\newcommand{\spa}{\ensuremath{\operatorname{span}}}

\newcommand{\Fix}{\ensuremath{\operatorname{Fix}}}
\newcommand{\Id}{\ensuremath{\operatorname{Id}}}

\newcommand{\jj}{\ensuremath{\,\mathfrak{q}}}

\newcommand{\bx}{\ensuremath{\mathbf{x}}}
\newcommand{\ba}{\ensuremath{\mathbf{a}}}

\newcommand{\bT}{\ensuremath{{\mathbf{T}}}}
\newcommand{\wT}{\ensuremath{{\widetilde{\mathbf{T}}}}}

\newcommand{\bL}{\ensuremath{{\mathbf{L}}}}

\newcommand{\bS}{\ensuremath{{\mathbf{S}}}}
\newcommand{\bD}{\ensuremath{{\mathbf{D}}}}
\newcommand{\bA}{\ensuremath{{{\mathbf{A}}}}}

\newcommand{\bB}{\ensuremath{{\mathbf{B}}}}

\newcommand{\bee}{\ensuremath{\mathbf{e}}}
\newcommand{\bC}{\ensuremath{\mathbf{C}}}
\newcommand{\by}{\ensuremath{\mathbf{y}}}

\newcommand{\sm}{\ensuremath{\mathbb{S}^{N\times N}}}
\newcommand{\smp}{\ensuremath{\mathbb{S}^{N\times N}_{+}}}
\newcommand{\smpp}{\ensuremath{\mathbb{S}^{N\times N}_{++}}}

{\begin{list}{}{%
\settowidth{\labelwidth}{\textrm{#1~}}%
\setlength{\leftmargin}{\labelwidth+\labelsep}}}
{\end{list}}
\newtheorem{theorem}{Theorem}[section]

\newtheorem{corollary}[theorem]{Corollary}
\newtheorem{proposition}[theorem]{Proposition}
\newtheorem{definition}[theorem]{Definition}
\theoremstyle{plain}{\theorembodyfont{\rmfamily}
}
\theoremstyle{plain}{\theorembodyfont{\rmfamily}
}
\theoremstyle{plain}{\theorembodyfont{\rmfamily}
}
\theoremstyle{plain}{\theorembodyfont{\rmfamily}
\newtheorem{example}[theorem]{Example}}
\newtheorem{fact}[theorem]{Fact}
\theoremstyle{plain}{\theorembodyfont{\rmfamily}
\newtheorem{bahy}[theorem]{Basic Hypothesis}}
\theoremstyle{plain}{\theorembodyfont{\rmfamily}
\newtheorem{remark}[theorem]{Remark}}

\newcommand{\boxedeqn}[1]{%
    \[\fbox{%
        \addtolength{\linewidth}{-2\fboxsep}%
        \addtolength{\linewidth}{-2\fboxrule}%
        \begin{minipage}{\linewidth}%
        \begin{equation}#1\\[+4mm]\end{equation}%
        \end{minipage}%
      }\]%
  }

\allowdisplaybreaks 

\begin{document}

\title{\textrm{On moving averages}}

\author{
Heinz H.\ Bauschke\thanks{
Mathematics, University
of British Columbia,
Kelowna, B.C.\ V1V~1V7, Canada.
E-mail: \texttt{heinz.bauschke@ubc.ca}.},~
Joshua Sarada\thanks{
Mathematics, University of
British Columbia,
Kelowna, B.C.\ V1V~1V7, Canada. E-mail:
\texttt{jshsarada@gmail.com}.},
~and~
Xianfu\ Wang\thanks{Mathematics,
University of British Columbia,
Kelowna, B.C.\ V1V~1V7, Canada.
E-mail:  \texttt{shawn.wang@ubc.ca}.}
}

\date{June 15, 2012}
\maketitle

\vskip 8mm

\begin{abstract} \noindent
We show that the moving arithmetic average is closely connected to a 
Gauss--Seidel type fixed point method studied by Bauschke, Wang and Wylie,
and which was observed to converge only numerically. 
Our analysis establishes a rigorous
proof of convergence of their algorithm in a special case;
moreover, limit is explicitly identified.  
Moving averages in Banach spaces
and Kolmogorov means are also studied. Furthermore, 
we consider moving proximal averages and epi-averages of convex functions.
\end{abstract}

{\small
\noindent
{\bfseries 2010 Mathematics Subject Classification:}
{Primary 15B51, 26E60, 47H10;
Secondary 39A06, 65H04, 47J25, 49J53. 
}

\noindent {\bfseries Keywords:}
Arithmetic mean, difference equation, 
epi-average, Kolmogorov mean, linear recurrence relation, 
means, moving average,
proximal average, stochastic matrix.
}

\section{Introduction}

Throughout this paper, we assume that $m\in\{2,3,\ldots\}$ and that 
\boxedeqn{
\text{$X=\RR^m$ is
an $m$-dimension real Euclidean space with standard inner
product $\scal{\cdot}{\cdot}$}
}
and induced norm $\|\cdot\|$.
We also assume that
\boxedeqn{\label{e:alphanumb}
\text{$\alpha_{0},\ldots, \alpha_{m-1}$ are nonnegative real numbers such
that $\sum_{i=0}^{m-1}\alpha_{i}=1$, and $\alpha_m := \alpha_0$. }
}
It will be convenient to introduce the following notation for the partial
sums and associated weights:
\boxedeqn{
\big(\forall k\in\{0,1,\ldots,m-1\}\big)\quad
a_k = \sum_{i=0}^k \alpha_i
\;\;\text{and}\;\;\ba = (a_0,a_1,\ldots,a_{m-1})^*.
}
and
\boxedeqn{
\big(\forall k\in \{0,1,\ldots,m-1\}\big)\quad 
\lambda_k 
= \frac{\ba^*\bee_k}{\ba^*\bee}
=\frac{a_k}{\sum_{i=0}^{m-1}a_i}
=
\frac{\sum_{i=0}^{k}\alpha_i}{\sum_{j=0}^{m-1}\sum_{i=0}^{j}\alpha_i}
\;\text{and}\;
\boldsymbol{\lambda}=
\begin{pmatrix}
\lambda_0\\
\vdots\\
\lambda_{m-1}
\end{pmatrix}. 
}
We shall study the homogeneous linear difference equation 
\begin{equation}\label{e:intro}
(\forall n\geq m)\quad 
y_{n}=\alpha_{m-1}y_{n-1}+\cdots+ \alpha_{0}y_{n-m},
\;\;\text{where}\;\; (y_{0},\ldots, y_{m-1})\in \RR^m.
\end{equation}
In the literature, this is called the 
\emph{moving average} (and it is also known as the 
rolling average, rolling mean or running average). 
It says that given a series of numbers and a fixed subset size, the
first element of the moving average is obtained by taking the average
of the initial fixed subset of the number series. Then the subset
is modified by "shifting forward", excluding the first number of
the series and including the next number following the original
subset in the series. This creates a new subset of numbers, which
is averaged. This process is repeated over the entire data series.
The moving average is widely applied in statistics, signal processing,
econometrics and mathematical finance; see, e.g., also
\cite{BBS,BorBor,combettes02,BoxJen,Tsokos}.

In \cite{BWW,BWWSIOPT}, we observed the numerical convergence of a
Gauss--Seidel type fixed point iteration numerically, but we were unable to
provide a rigorous proof.
In this note, we present a connection between the moving average and this 
fixed point recursion; this allows us to give an analytical proof for the 
case when all monotone operators are zero.
Moreover, the limit is identified. However this approach is unlike to
generalize to the general fixed point iteration due to the interlaced
nonlinear resolvents. 

While the results rest primarily on results from linear algebra, we
consider in the second half of the paper several related highly nonlinear
moving averages that exhibit a ``hidden linearity'' and thus allow to be
rigorously studied.

The paper is organized as follows. 
In Section~\ref{s:facts}, we present various facts and auxiliary results
rooted ultimately in Linear Algebra. In
Section~\ref{s:mainresults} we establish the connection between the
moving average and the Gauss--Seidel type iteration scheme studied
by Bauschke, Wang and Wylie, which
 leads to a rigorous proof for the convergence of their algorithm in the
 aforementioned special case. 
In Section~\ref{s:specialcase} we show that the iteration
 matrix has a closed form when $\alpha_{0}=0$ and 
$\alpha_{1}=\cdots=\alpha_{m-1}= 1/(m-1)$. 
Moving Kolmogorov means (also known as $f$-means) 
are considered in Section~\ref{s:fmeans}. In particular, various
known means such as arithmetic mean, harmonic mean, resolvent mean, etc.\ 
\cite{BorBor,bauschke2010} all turn out to be special cases of
this general framework. The final Section~\ref{s:lastsection}
concerns moving proximal averages and epi-averages of convex functions.

Our notation follows \cite{BC2011,Meyer,Rock70,Rock98}. 
The 
\emph{identity operator} $\Id$ is defined by 
$X\to X\colon x\mapsto x$. 
A mapping $T\colon X\to X$ is \emph{nonexpansive} (Lipschitz-$1$) 
if $(\forall x\in X)(\forall y\in X)$ $\|Tx-Ty\|\leq\|x-y\|$. 
The set of \emph{fixed points} of $T$ is denoted by
$\Fix T = \menge{x\in X}{x=Tx}$.  
The following matrix
plays a central role in this paper
\boxedeqn{
\label{e:companionm}
\bA =
\begin{pmatrix}
0 & 1 & 0& \ldots  & 0\\
0 & 0 & 1 & \cdots & 0\\
\vdots & \vdots & \ddots &\ddots &\vdots \\
 0 & 0 & 0 &\cdots & 1\\
\alpha_{0} & \alpha_{1} & \alpha_{2} & \cdots & \alpha_{m-1}
\end{pmatrix}.
}
Note that 
\begin{equation}
\label{e:FixA}
\Fix\bA = \bD=\mmenge{\bx=(x,x,\ldots,x)\in X}{x\in \RR}, 
\end{equation}
where $\bD$ is also called the \emph{diagonal} in $X$. 
By \cite[page~648]{Meyer}, 
the characteristic polynomial of $\bA$ is 
\begin{equation}
p(x)=x^{m}-\alpha_{m-1}x^{m-1}-\cdots-\alpha_{1}x-\alpha_{0},
\end{equation}
and, in turn, $\bA$ is the (transpose of the) companion matrix of $p(x)$. 

As usual, we use $\CC$ for the field of complex numbers;
$\RR_{+}$ ($\RR_{++}$) for the nonnegative real numbers (positive real numbers);
$\NN = \{0,1,2,\ldots\}$ for the natural numbers. 
For $\bB \in \CC^{m\times m}$, 
$\rho(\bB)=\max_{\lambda\in\sigma(\bB)}|\lambda|$
is the spectral radius of $\bB$ and
$\sigma(\bB)$ denotes the \emph{spectrum} of $\bB$, i.e., the set of 
eigenvalues of $\bB$. 
Given $\lambda\in\sigma(\bB)$,
we say that $\lambda$ is \emph{simple} if its algebraic multiplicity is
$1$ and that it is \emph{semisimple} if its algebraic and geometric 
multiplicities coincide. 
A simple eigenvalue is always semisimple,
see \cite[pages~510--511]{Meyer}. 
A vector $\by\in \CC^m\smallsetminus \{0\}$ satisfying
$\bB\by=\lambda\by$ ($\by^{*}\bB=\lambda\by^{*}$) is called a 
\emph{right-hand (left-hand) eigenvector} of $\bB$ with respect to the
eigenvalue $\lambda$. (The $\,^*$ denotes the conjugate transpose of a
vector or a matrix.) 
Then \emph{range} of an operator $\bB$ is denoted by $\ran(\bB)$;
if $\bB$ is linear, we denote its \emph{kernel} by $\ker(\bB)$. 
It will be convenient to denote the $i$th standard unit column vector by
$\bee_i$ and to also set 
$\bee=\sum_{i=1}^{m}\bee_i= (1,1,\ldots,1)^*\in X$.
We denote by $\sm$ the space of all $N\times N$ real symmetric matrices, 
while $\smp$ and $\smpp$ stand, respectively, for the
set of $N\times N$ positive semidefinite and positive definite matrices.
The \emph{greatest common divisor} of a set of integers $S$ is denoted
by $\gcd(S)$. 
Turning to functions, we let $\GX(X)$ be the set of all functions
that are convex, lower semicontinuous, and proper.
We denote the quadratic energy function by
$ \jj = \thalb\|\cdot\|^2$. 
Finally, the \emph{Fenchel conjugate} $g^*$ of a function $g$ is 
given by 
$g^{*}(x)=\sup_{y\in X}(\langle x,y\rangle-g(y))$.

\section{Linear algebraic results}

\label{s:facts}

To make our analysis self-contained, 
let us gather in this section some useful facts and
auxiliary results on stochastic matrices, linear recurrence relations 
and the convergence of matrix powers.

\subsection*{Basic properties of stochastic matrices}

Recall that a matrix $\bB$ with nonnegative entries is called
\emph{stochastic} (or row-stochastic) 
if each row sum is equal to $1$.
(See \cite[Chapter~8]{berman} or \cite[Section~8.4]{Meyer} for more on
stochastic matrices.)

\begin{fact}
{\rm (See \cite[pages~689 and 696]{Meyer}.)}
\label{f:rootone}
Let $\bB\in\RR^{m\times m}$ be stochastic. 
Then $\rho(\bB)=1$ and $1$ is a semisimple eigenvalue of $\bB$
with eigenvector $\bee$. 
\end{fact}

The proof of the next result is a simple verification and hence omitted.

\begin{proposition}\label{p:lefteigenv}
The vectors $\ba^*$ and $\bee$ are, respectively, 
left-hand and right-hand eigenvectors of $\bA$ 
associated with the eigenvalue $1$. 
\end{proposition}

\subsection*{Linear recurrence relations and polynomials}

Consider the linear recurrence relation 
\begin{equation}
\label{e:lrecursion}
(\forall n\geq m)\quad 
y_{n}=\alpha_{m-1}y_{n-1}+\cdots+ \alpha_{0}y_{n-m},
\;\;\text{where}\;\; (y_{0},\ldots, y_{m-1})\in \RR^m.
\end{equation}
Setting
$(\forall\nnn)$ $\by^{[n]} = (y_n,y_{n+1},\ldots,y_{n+m-1})^*$,
we see that we can rewrite 
\eqref{e:lrecursion} as
\begin{equation}
\label{e:therewrite}
(\forall\nnn)\quad \by^{[n+1]}=\bA\by^{[n]} = \cdots =
\bA^{n+1}\by^{[0]}.
\end{equation}
This explains our interest in understanding the limiting behaviour of
powers of $\bA$, which yield information about the limiting behaviour
of $(\by^{[n]})_\nnn$ and hence of $(y_n)_\nnn$. 
The solution to \eqref{e:lrecursion} depends on the roots of
characteristic polynomial
\begin{equation}\label{e:characteristic}
p(x)=x^{m}-\alpha_{m-1}x^{m-1}-\cdots-\alpha_{1}x-\alpha_{0}
\end{equation}
of $\bA$.

\begin{fact}
{\rm (See \cite[page~90]{Ostrowski} or \cite[pages~74--75]{Ortega})}
Let $k$ be the number of distinct roots $u_1,\ldots,u_k$ of
\eqref{e:characteristic} with multiplicities 
$m_1,\ldots,m_k$, respectively. 
Then for each $i\in \{1,\ldots,k\}$, there exists
a polynomial $q_i$ of degree $m_i-1$ such that 
the general solution of \eqref{e:lrecursion} is
\begin{equation}
(\forall\nnn)\quad y_{n}=\sum_{i=1}^{k}q_i(n)u_{i}^{n}.
\end{equation}
Consequently, if $k=m$, i.e., all roots are distinct, then 
there exists $(\nu_1,\ldots,\nu_m)\in\CC^m$ such that 
\begin{equation}
(\forall\nnn)\quad y_{n}=\sum_{i=1}^{m}\nu_iu_{i}^{n}.
\end{equation}
\end{fact}

\begin{fact}[Ostrowski]
{\rm (See \cite[Theorem~12.2]{Ostrowski}.)}
\label{f:simpleroot} 
Let $(b_1,\ldots,b_m)\in\RP^m$ and set 
\begin{equation}\label{Cauchy}
q(x) = x^{m}-b_{1}x^{m-1}-\cdots-b_{m}. 
\end{equation}
Suppose that $\gcd\menge{k\in\{1,\ldots,m\}}{b_k>0}=1$.
Then $q$ has a unique positive root $r$. Moreover,
$r$ is a simple root of $q$, and the modulus of every other root of $q$ is
strictly less than $r$. 
\end{fact}

See also \cite{BorErd,Prasolov} for further results on polynomials. 
Let us now state a basic assumption that we will impose repeatedly. 

\begin{bahy}
\label{bahy}
The polynomial $p(x)$ given by \eqref{e:characteristic}
has a unique positive root, which is simple and equal to $1$,
and each other root has modulus strictly less than $1$.
This happens if one of the following holds:
\begin{enumerate}
\item
\label{bahy1}
$\gcd\menge{i\in\{1,\ldots,m\}}{\alpha_{m-i}>0} = 1$. 
\item
\label{bahy2}
$\alpha_{m-1}>0$. 
\item
\label{bahy3}
$(\exi i\in\{1,\ldots,m-1\})$ $\alpha_{m-i}\alpha_{m-(i+1)}>0$. 
\item
\label{bahy4}
$(\exi \{i,j\}\subseteq \{1,\ldots,m-1\})$ $\gcd\{i,j\}=1$ and
$\alpha_{m-i}\alpha_{m-j}>0$. 
\end{enumerate}
\end{bahy}
\begin{proof}
\ref{bahy1}: This is clear from Fact~\ref{f:simpleroot} and the definition of
$p(x)$. 
\ref{bahy2}--\ref{bahy4}: Each condition implies \ref{bahy1}. 
\end{proof}

\subsection*{Limits of matrix powers and linear recurrence relations}

\begin{fact}
{\rm (See \cite[pages~383--386, 518 and 630]{Meyer}.)} 
\label{f:limit}
Let $\bB\in \CC^{m\times m}$. 
Then the following hold:
\begin{enumerate}
\item 
$(\bB^n)_\nnn$ converges to a nonzero matrix if and only if 
$1$ is a semisimple eigenvalue of $\bB$ and every other eigenvalue of
$\bB$ has modulus strictly less than $1$. 
\item 
If $\lim_\nnn \bB^n$ exists, then 
is the projector onto $\ker(\Id-\bB)$ along $\ran(\Id-\bB)$.
\item
\label{f:limit3}
If $\rho(\bB)=1$ is a simple eigenvalue with right-hand and left-hand eigenvectors
$x$ and $y^*$ respectively, then 
$\displaystyle \lim_{\nnn} \bB^n = \frac{xy^*}{y^*x}$.  
\end{enumerate}
\end{fact}

\begin{corollary}
\label{c:limbA}
Suppose that the {\rm\bf Basic~Hypothesis~\ref{bahy}} holds. 
Then 
\begin{equation}
\lim_{\nnn}\bA^n = \frac{\bee\ba^*}{\ba^*\bee}=
\frac{1}{\sum_{k=0}^{m-1}\sum_{i=0}^{k}\alpha_{i}}
\begin{pmatrix}
\alpha_{0} &\sum_{i=0}^{1}\alpha_{i}& \cdots &  \sum_{i=0}^{m-2}\alpha_{i} & 1\\
\alpha_{0} &\sum_{i=0}^{1}\alpha_{i}& \cdots &  \sum_{i=0}^{m-2}\alpha_{i} & 1\\
\vdots & & &  & &\\
\alpha_{0} &\sum_{i=0}^{1}\alpha_{i}& \cdots &  \sum_{i=0}^{m-2}\alpha_{i} & 1
\end{pmatrix}
=\bee\boldsymbol{\lambda}^*.
\end{equation}
\end{corollary}
\begin{proof}
Clear from Proposition~\ref{p:lefteigenv}, 
Fact~\ref{f:limit}\ref{f:limit3} and the definitions.
\end{proof}

\begin{corollary}\label{c:positivecoe}
Suppose that the {\rm\bf Basic~Hypothesis~\ref{bahy}} holds and
consider sequence $(y_n)_\nnn$ generated by the linear recurrence relation 
\begin{equation}
(\forall n\geq m)\quad 
y_{n}=\alpha_{m-1}y_{n-1}+\cdots+ \alpha_{0}y_{n-m},
\;\;\text{where}\;\; \by = (y_{0},\ldots, y_{m-1})\in \RR^m.
\end{equation}
Then 
\begin{equation}
\lim_\nnn y_n = \frac{\ba^*\by}{\ba^*\bee} =
\frac{\sum_{k=0}^{m-1}a_ky_k}{\sum_{k=0}^{m-1}a_k } = 
\frac{\sum_{k=0}^{m-1}\sum_{i=0}^{k}\alpha_i
y_k}{\sum_{k=0}^{m-1}\sum_{i=0}^{k}\alpha_i}
=\sum_{k=0}^{m-1}\lambda_ky_k.
\end{equation}
\end{corollary}
\begin{proof}
This follows from Corollary~\ref{c:limbA} and \eqref{e:therewrite}. 
\end{proof}


\begin{definition}
\label{d:ell}
Let $Y$ be a real Banach space and consider
the sequence $(y_n)_\nnn$ in $Y$
generated by the linear recurrence relation 
\begin{equation}
(\forall n\geq m)\quad 
y_{n}=\alpha_{m-1}y_{n-1}+\cdots+ \alpha_{0}y_{n-m},
\;\;\text{where}\;\; \by = (y_{0},\ldots, y_{m-1})\in Y^m.
\end{equation}
If $\lim_\nnn y_n$ exists, then we set it equal to 
$\ell(\by)$ and we write $\by\in\dom\ell$.
\end{definition}

\begin{remark}
\label{r:0614b} 
Note that $\dom\ell$ is a linear subspace of $Y^m$,
that $\menge{\by=(y,\ldots,y)\in Y^m}{y\in Y}\subseteq\ell$, 
and that $\ell\colon \dom\ell\to Y$ is a linear operator.
Furthermore,
\begin{equation}
\big(\forall (y_{0},\ldots, y_{m-1})\in Y^m\big)\quad
\{y_n\}_\nnn \subseteq \conv\{y_0,\ldots,y_{m-1}\}
\subseteq \spa\{y_0,\ldots,y_{m-1}\}.
\end{equation}
This implies
\begin{equation}
\big(\forall \by=(y_{0},\ldots, y_{m-1})\in \dom\ell)\quad
\ell (\by)\in \conv\{y_0,\ldots,y_{m-1}\}
\subseteq \spa\{y_0,\ldots,y_{m-1}\}
\end{equation}
because the convex hull of finitely many points is compact
(see, e.g., \cite[Corollary~5.30]{AliBor}). 
\end{remark}

We are now ready to lift Corollary~\ref{c:positivecoe} to general
Banach spaces.

\begin{corollary}\label{c:generalcase}
Suppose that the {\rm\bf Basic~Hypothesis~\ref{bahy}} holds and
consider sequence $(y_n)_\nnn$ generated by the linear recurrence relation 
\begin{equation}
\label{e:general}
(\forall n\geq m)\quad 
y_{n}=\alpha_{m-1}y_{n-1}+\cdots+ \alpha_{0}y_{n-m},
\;\;\text{where}\;\; \by = (y_{0},\ldots, y_{m-1})\in Y^m,
\end{equation}
where $Y$ is a real Banach space. 
Then 
\begin{equation}
\label{e:0614a}
\lim_\nnn y_n = 
\ell(y_0,\ldots,y_{m-1}) = 
\frac{\sum_{k=0}^{m-1}a_ky_k}{\sum_{k=0}^{m-1}a_k } = 
\frac{\sum_{k=0}^{m-1}\sum_{i=0}^{k}\alpha_i
y_k}{\sum_{k=0}^{m-1}\sum_{i=0}^{k}\alpha_i}
=\sum_{k=0}^{m-1}\lambda_ky_k.
\end{equation}
\end{corollary}
\begin{proof}
Denote the right-hand side of \eqref{e:0614a} by $z$, 
let $y^*\in Y^*$ and define $(\forall\nnn)$ 
$\eta_n=y^*(y_n)$. Applying $y^*$ to \eqref{e:general} gives rise
to the linear recurrence relation 
\begin{subequations}
\begin{equation}
\label{e:shadow1}
(\forall n\geq m)\quad 
\eta_{n}=\alpha_{m-1}\eta_{n-1}+\cdots+ \alpha_{0}\eta_{n-m},
\end{equation}
where 
\begin{equation}
\label{e:shadow2}
(\eta_{0},\ldots, \eta_{m-1})=\big(y^*(y_0),\ldots,y^*(y_{m-1})\big)\in \RR^m.
\end{equation}
\end{subequations}
Corollary~\ref{c:positivecoe} and the linearity and continuity of $y^*$ 
imply that 
\begin{equation}
y^*(y_n) = \eta_n \to \frac{\sum_{k=0}^{m-1}a_k\eta_k}{\sum_{k=0}^{m-1}a_k }
= y^*\bigg(\frac{\sum_{k=0}^{m-1}a_ky_k}{\sum_{k=0}^{m-1}a_k }\bigg) =
y^*(z). 
\end{equation}
It follows that $(y_n)_\nnn$ converges weakly to $z$. 
On the other hand, $(y_n)_\nnn$ lies not only in a compact subset but also
in a finite-dimensional subspace of $Y$ (see Remark~\ref{r:0614b}). 
Altogether, we deduce that $(y_n)_\nnn$ strongly converges to $z$. 
\end{proof}

\subsection*{Reducible matrices}

Recall (see, e.g., \cite[page~671]{Meyer}) 
that $\bB\in \CC^{m\times m}$ is \emph{reducible} 
if there exists a permutation matrix
$P$ such that
\begin{equation}
P^{*}\bB P=\begin{pmatrix}
U & W \\
0 & V
\end{pmatrix},
\end{equation}
where $U$ and $V$ are nontrivial square matrices;
otherwise, $\bB$ is \emph{irreducible}.

\begin{proposition}\label{p:reducible}
Let $\bB\in\CC^{m\times m}$ have at least one zero column. 
Then $\bB$ is reducible.
\end{proposition}
\begin{proof} 
By assumption, there exists a permutation matrix $P$ 
such that the first column of $\bB P$ is equal to the zero vector. 
Then 
\begin{equation}
P^{*}\bB P=\begin{pmatrix}
U & W \\
0 & V
\end{pmatrix},
\end{equation}
where $U=0\in\CC^{1\times 1}$ and 
$V\in \CC^{(m-1)\times (m-1)}$ are nontrivial square matrices.
\end{proof}

\subsection*{Circulant matrices}

It is interesting to compare the above results to circulant matrices.
To this end, we set
\begin{equation}
\bC=\begin{pmatrix}
\alpha_{0} &\alpha_{1}& \alpha_{2} & \cdots & \alpha_{m-1}\\
\alpha_{m-1} & \alpha_{0} &\alpha_{1} & \cdots & \alpha_{m-2}\\
\alpha_{m-2} &\alpha_{m-1} &\alpha_{0} &\cdots &\alpha_{m-3}\\
& \cdots & & \vdots \\
\alpha_{1} &\alpha_{2}& \alpha_{3}& \cdots &\alpha_{0}
\end{pmatrix}.
\end{equation}

\begin{fact}
{\rm (See \cite[Theorems 1 and 2]{Krafft}.)}
\label{f:kraftsh}
Set $U = \menge{k\in\{1,\ldots,m\}}{\alpha_{k}>0}$. 
Then
\begin{equation}
\bL = \lim_\nnn \bC^n \text{~exists}
\end{equation}
if and only if $\gamma = \gcd(U\cup\{m\}) = \gcd((U-U)\cup\{m\})$,
in which case the entries of $\bL$ satisfy 
\begin{equation}
\bL_{i,j} = \begin{cases}
\gamma/m, &\text{if $i\equiv j\pmod{\gamma}$;}\\
0, &\text{otherwise.}
\end{cases}
\end{equation}
\end{fact}

\begin{remark}
Note that $\bC$ is \emph{bistochastic}, i.e., both $\bC$ and $\bC^*$ are
stochastic. Since permutation matrices are clearly nonexpansive,
it follows from Birkhoff's theorem
(see, e.g., \cite[Theorem~8.7.1]{Horn})
that $\bC$ is convex combination of nonexpansive matrices.
Hence, $\bC$ is nonexpansive as well. 
Moreover, one verifies readily that $\Fix\bC = \bD$.
\end{remark}

\begin{example}
Suppose there exists $i\in\{1,\ldots,m-1\}$ such that
$\{i,i+1\}\subseteq U$. Then $(\bC^n)_\nnn$ converges to the 
orthogonal projector onto $\bD$, i.e., 
\begin{equation}
\lim_\nnn \bC^n = 
\frac{1}{m}\bee \bee^{*}
=\frac{1}{m}\begin{pmatrix}
1 & 1& \cdots &1 \\
1 & 1 & \cdots & 1\\
\vdots & &  & \\
1 & 1& \cdots &1
\end{pmatrix}.
\end{equation}
\end{example}

For further results on circulant matrices, see  \cite{chou} 
and in particular \cite{Davis}.

\section{Main results}\label{s:mainresults}

In this section, we study the convergence of the 
Gauss--Seidel type fixed point iteration
scheme proposed in \cite{BWW,BWWSIOPT}.
Recalling \eqref{e:alphanumb}, we set 
\begin{subequations}
\label{e:defofTk}
\begin{align}
\bT_1 &=
\begin{pmatrix}
\alpha_{0} & \alpha_{1} & \alpha_{2} & \cdots & \alpha_{m-1}\\
0 & 1 & 0 & \cdots & 0\\
0 & 0 & 1 & \cdots & 0\\
\vdots & & & & \\
0 & 0 & 0 & \cdots & 1
\end{pmatrix},\quad
\bT_2 =
\begin{pmatrix}
1 & 0 & 0 & \cdots & 0\\
\alpha_{m-1} & \alpha_{0} & \alpha_{1} & \cdots & \alpha_{m-2}\\
0 & 0 & 1 & \cdots & 0\\
\vdots & & & & \\
0 & 0 & 0 & \cdots & 1
\end{pmatrix},\\
&\;\; \vdots\nonumber\\
\bT_m &=
\begin{pmatrix}
1 & 0 & 0 & \cdots & 0\\
0 & 1 & 0 & \cdots & 0\\
0 & 0 & 1 & \cdots & 0\\
\vdots & & & & \\
\alpha_{1} & \alpha_{2} & \alpha_{3} & \cdots & \alpha_{0}
\end{pmatrix};
\end{align}
\end{subequations}
or entrywise 
\begin{equation}
\big(\bT_k\big)_{i,j} =
\begin{cases}
\alpha_{m-k+j}, &\text{if $i=k$ and $1\leq j< k$;}\\
\alpha_{j-k}, &\text{if $i=k$ and $m\geq j\geq k$;}\\
1, &\text{if $i\neq k$ and $j=i$;}\\
0, &\text{otherwise.}
\end{cases}
\end{equation}
Furthermore, we define
\boxedeqn{\label{e:defofbT}\bT=\bT_{m}\cdots \bT_{1}.}

When $\alpha_{0}=0, \alpha_{1}=\cdots=\alpha_{m-1}=\frac{1}{m-1}$, 
iterating $\bT$ exactly corresponds to the
algorithm investigated in \cite{BWW, BWWSIOPT}
when all monotone operators are zeros. 
The authors there observed the numerical convergence but were not able
to provide a rigorous proofs. 
We shall present a rigorous proof by connecting $\bT$ and $\bA$.
We start with some basic properties.

\begin{proposition}
\label{p:tk}
Let $k\in\{1,\ldots,m\}$. 
Then the following hold:
\begin{enumerate}
\item\label{t:tk1}
$\bT_{k}$ and $\bT$ are stochastic.
\item\label{t:tk2}
If $\alpha_{0}=0$,
then neither $\bT_{k}$ nor $\bT$ is nonexpansive.
\item\label{t:ttk} 
If $\alpha_{0}=0$, then neither $\bT_{k}$ nor $\bT$
is irreducible.
\end{enumerate}
\end{proposition}
\begin{proof}
\ref{t:tk1}: 
By \eqref{e:alphanumb}, $\bT_{k}$ is stochastic and so is
therefore $\bT$ as a product of stochastic matrices. 

\ref{t:tk2}: 
Indeed, in this case $\bT_k(\bee-\bee_k) = \bee$ and
thus $\|\bT_k(\bee-\bee_k)\|^2=m>m-1 =\|\bee-\bee_k\|^2$. 
Similarly, $\bT(\bee-\bee_1)=\bee$ and 
thus $\|\bT(\bee-\bee_1)\|^2 = m > m-1 = \|\bee-\bee_1\|^2$.

\ref{t:ttk}: 
In this case, the $\bT_1$ and $\bT_k$ have a zero column, hence so does
$\bT$. Now apply Proposition~\ref{p:reducible}.
\end{proof}

Proposition~\ref{p:tk} illustrates neither the theory of nonexpansive
mappings nor that of irreducible matrices (as is done in, e.g.,
\cite{BBS,Meyer}) is applicable to study limiting
properties of $(\bT^n)_\nnn$. Fortunately, we are able to base our
analysis on the 
right-shift and left-shift operators which are respectively given by
\begin{equation}
R=\begin{pmatrix}
0 & 0& 0 & \cdots &  1\\
1 & 0 & 0 & \cdots & 0\\
0 & 1 & 0 &\cdots & 0\\
\vdots  & & \ddots & \ddots &\\
  0 & \cdots & 0 & 1 & 0
\end{pmatrix},
\qquad L=\begin{pmatrix}
0 & 1 & 0 & \cdots &  0\\
0 & 0 & 1 & \cdots & 0\\
\vdots &  &\ddots & \ddots & \vdots &\\
 0 &  0& 0& \cdots & 1\\
  1 & 0 & 0 & \cdots & 0
\end{pmatrix}.
\end{equation}
Note that $R$ and $L$ are permutation matrices which satisfy
the following properties which we will use repeatedly:
\begin{equation}
R^m=L^m=\Id,\; R^{*}=R^{-1}=L,\; L^{*}=L^{-1}=R,\;R^{m-k}=L^{k},\;
\text{where $k\in\{0,\ldots,m\}$. }
\end{equation}

A key observation is the following result which connects $\bT_{k}$ and the companion matrix
$\bA$.

\begin{proposition}\label{p:tkproduct} 
For every $k\in\{1,\ldots,m\}$, we have
\begin{subequations}
\begin{align}\label{e:inA}
\bT_{k}&=R^{k}\bA L^{k-1},\\
\label{e:tk}
\bT_{k}\bT_{k-1}\cdots \bT_{1}&=R^{k}\bA^{k},\\
\label{e:tm}
\bT&=\bA^{m}.
\end{align}
\end{subequations}
\end{proposition}
\begin{proof} 
We prove this by induction on $k$. 
Clearly, \eqref{e:inA} and \eqref{e:tk} hold when $k=1$.
Now assume that \eqref{e:inA} and \eqref{e:tk} hold for some
$k\in\{1,\ldots,m-1\}$. 
Then 
$\bT_{k+1}= R\bT_{k}L=R(R^{k}AL^{k-1})L=R^{k+1}AL^{k}$,
which is \eqref{e:inA} for $k+1$. 
Hence 
$\bT_{k+1}\cdots \bT_1 = \bT_{k+1}(\bT_{k}\cdots \bT_1) = 
(R^{k+1}\bA L^k)(R^k\bA^k)=R^{k+1}\bA^{k+1}$, which is \eqref{e:tk} for $k+1$.
Finally, \eqref{e:tm} follows from \eqref{e:defofbT} and \eqref{e:tk} with $k=m$. 
\end{proof}

We are now able to derive our main results which resolves a special case of
an open problem posed in \cite{BWWSIOPT}. 

\begin{theorem}[main result]
\label{t:tfix}
Suppose that the {\rm\bf Basic~Hypothesis~\ref{bahy}} holds. 
Then
\begin{equation}
\label{e:thelimit}
\lim_{\nnn}\bT^n = \frac{\bee\ba^*}{\ba^*\bee} = 
\frac{1}{\sum_{k=0}^{m-1}\sum_{i=0}^{k}\alpha_{i}}
\begin{pmatrix}
\alpha_{0} &\sum_{i=0}^{1}\alpha_{i}& \cdots &  \sum_{i=0}^{m-2}\alpha_{i} & 1\\
\alpha_{0} &\sum_{i=0}^{1}\alpha_{i}& \cdots &  \sum_{i=0}^{m-2}\alpha_{i} & 1\\
\vdots & & &  & &\\
\alpha_{0} &\sum_{i=0}^{1}\alpha_{i}& \cdots &  \sum_{i=0}^{m-2}\alpha_{i} & 1
\end{pmatrix} = \bee\boldsymbol{\lambda}^*. 
\end{equation}
and 
$\Fix\bT=\bD = \Fix\bT_1\cap\cdots\cap \Fix\bT_m$.
\end{theorem}
\begin{proof}
In view of \eqref{e:tm}, it is clear that $(\bT^n)_\nnn$ is a
subsequence of $(\bA^n)_\nnn$. 
Hence \eqref{e:thelimit} follows from Corollary~\ref{c:limbA}. 

Now set $\bL = \lim_\nnn \bT^n = \ba^*\bee/(\bee^*\ba)$. 
On the one hand, $\bD = \Fix \bA \subseteq \Fix \bT$
by \eqref{e:FixA} and \eqref{e:tm}. 
On the other hand, 
$\Fix\bT \subseteq\ran\bL\subseteq\bD$. 
Altogether, $\Fix\bT=\bD$.
Finally, clearly $\Fix\bT_1\cap\cdots \Fix\bT_m\subseteq \Fix\bT =\bD
\subseteq\Fix\bT_1\cap \cdots\cap\Fix\bT_m$. 
\end{proof}

\begin{remark}
Theorem~\ref{t:tfix}, with the choice
$(\alpha_0,\ldots,\alpha_{m-1}) = (m-1)^{-1}(0,1,1,\ldots,1)$, settles
\cite[Questions \textbf{Q1} and \textbf{Q2} on p.~36]{BWWSIOPT} 
when all resolvents coincide with $\Id$.
Furthermore, \eqref{e:thelimit} gives an \emph{explicit formula} for the limit
of $(\bT^n)_\nnn$. 
\end{remark}

\section{A special case: $\alpha_{0}=0$ and $\alpha_{1}=\cdots=\alpha_{m-1}=1/(m-1)$}\label{s:specialcase}

This assignment of parameters was the setting of \cite{BWWSIOPT}.
In this case, even closed forms for $\bT$,
and for the partial products
$\bT_{k}\bT_{k-1}\cdots\bT_{1}$ and $A^{k}$, are available
as we demonstrate in this section.
To this end, we abbreviate 
\begin{equation}
\gamma = \frac{1}{m-1}.
\end{equation}
Then \eqref{e:defofTk} and \eqref{e:companionm} turn into 
\begin{subequations}
\begin{align}
\bT_1 &=
\begin{pmatrix}
0 & \gamma & \gamma & \cdots & \gamma\\
0 & 1 & 0 & \cdots & 0\\
0 & 0 & 1 & \cdots & 0\\
\vdots & & & & \\
0 & 0 & 0 & \cdots & 1
\end{pmatrix},\qquad
\bT_2 =
\begin{pmatrix}
1 & 0 & 0 & \cdots & 0\\
\gamma & 0 & \gamma & \cdots & \gamma\\
0 & 0 & 1 & \cdots & 0\\
\vdots & & & & \\
0 & 0 & 0 & \cdots & 1
\end{pmatrix},\\
& \ldots,\nonumber\\
\bT_m &=
\begin{pmatrix}
1 & 0 & 0 & \cdots & 0\\
0 & 1 & 0 & \cdots & 0\\
0 & 0 & 1 & \cdots & 0\\
\vdots & & & & \\
\gamma & \gamma & \gamma & \cdots & 0
\end{pmatrix}, 
\text{~and~}
\bA =
\begin{pmatrix}
0 & 1 & 0 & \cdots & 0\\
0 & 0 & 1 & \cdots & 0\\
0 & 0 & 0 & \cdots & 0\\
\vdots & & & & \\
0 & \gamma& \gamma & \cdots & \gamma
\end{pmatrix}.
\end{align}
\end{subequations}

\begin{proposition}
\label{l:1}
Define $\wT\in\RR^{m\times m}$ entrywise by
\begin{equation}
\label{e:T}
\wT_{i,j} :=
\begin{cases}
\displaystyle \gamma (1+\gamma)^{i-1}, &\text{if $i<j$;}\\[+5 mm]
\displaystyle \gamma(1+\gamma)^{i-1}- \gamma (1+\gamma)^{i-j}, &
\text{if $i\geq j$.}
\end{cases}
\end{equation}
Let $k\in\{1,\ldots,m\}$. Then 
\begin{equation}
\label{e:Sk}
\bT_k\bT_{k-1}\cdots\bT_1 =
\begin{pmatrix}
\text{first $k$ rows of $\wT$}\\[+2mm]
0_{(m-k)\times k} \quad I_{m-k}
\end{pmatrix};
\end{equation}
or entrywise
\begin{equation}
\label{e:0416b}
\big(\bT_k\bT_{k-1}\cdots\bT_1\big)_{i,j} :=
\begin{cases}
\wT_{i,j}, &\text{if $1\leq i \leq k$;}\\
1, &\text{if $k+1\leq i\leq m$ and $j=i$;}\\
0, &\text{if $k+1\leq i\leq m$ and $j\neq i$.}
\end{cases}
\end{equation}
In particular,
\begin{equation}
\bT=\wT.
\end{equation}
\end{proposition}
\begin{proof}
Set 
\begin{equation}
\bS_k = \bT_k\bT_{k-1}\cdots\bT_1.
\end{equation}
We prove \eqref{e:Sk} by induction on $k\in\{1,\ldots,m\}$;
the rest follows readily. 

$k=1$: On the one hand, by definition,
\begin{equation}
\bT_1 =
\begin{pmatrix}
0 & \gamma & \gamma & \cdots & \gamma\\
0 & 1 & 0 & \cdots & 0\\
0 & 0 & 1 & \cdots & 0\\
\vdots & & & & \\
0 & 0 & 0 & \cdots & 1
\end{pmatrix}.
\end{equation}
On the other hand, the first row of $\wT$ is, using \eqref{e:T},
is $(0,\gamma,\ldots,\gamma)$.
Hence the statement is true for $k=1$.

$k-1\leadsto k$:
Assume the statement is true for $k\in\{2,\ldots,m\}$, i.e.,
\begin{equation}
\bT_{k-1}\cdots\bT_1
= \begin{pmatrix}
\text{first $k-1$ rows of $\wT$}\\[+2mm]
0_{(m-(k-1))\times (k-1)} \quad I_{m-(k-1)}
\end{pmatrix}
= \begin{pmatrix}
\text{first $k-1$ rows of $\wT$}\\[+2mm]
0_{(m-k+1))\times (k-1)} \quad I_{m-k+1};
\end{pmatrix}
\end{equation}
or entrywise, this matrix $\bS$ satisfies
\begin{equation}
(\forall i)(\forall j)\quad
\bS_{i,j} :=
\begin{cases}
\wT_{i,j}, &\text{if $1\leq i \leq k-1$;}\\
1, &\text{if $k\leq i\leq m$ and $j=i$;}\\
0, &\text{if $k\leq i\leq m$ and $j\neq i$.}
\end{cases}
\end{equation}
Recall that
\begin{equation}
(\forall i)(\forall j)\quad
\big(\bT_k\big)_{i,j} =
\begin{cases}
\gamma, &\text{if $i=k$ and $j\neq k$;}\\
1, &\text{if $i\neq k$ and $j=i$;}\\
0, &\text{otherwise.}
\end{cases}
\end{equation}
It is clear that $\bT_k\bS =\bT_{k}\bT_{k-1}\cdots\bT_1$.
We must show that
\begin{equation}
\bT_k\bS \stackrel{?}{=}
\begin{pmatrix}
\text{first $k$ rows of $\wT$}\\[+2mm]
0_{(m-k)\times k} \quad I_{m-k}
\end{pmatrix}.
\end{equation}
We do this entrywise, and thus fix $i$ and $j$ in $\{1,\ldots,m\}$.

\emph{Case~1:} $1\leq i\leq k-1$.\\
Then $(\bT_k\bS)_{i,j} = \sum_{l=1}^{m}(\bT_k)_{i,l}\bS_{l,j} =
(\bT_k)_{i,i}\bS_{i,j} = \bS_{i,j}$, which shows that the first $k-1$
rows of $\bT_k\bS$ are the same as the first $k-1$ rows of $\bS$, which
in turn are the same as the first $k-1$ rows of $\wT$, as required.

\emph{Case~2:} $i=k$.\\
Then $(\bT_k\bS)_{i,j} = (\bT_k\bS)_{k,j} =
\sum_{l=1}^{m}(\bT_k)_{k,l}\bS_{l,j} =
\sum_{l=1}^{k-1}(\bT_k)_{k,l}\bS_{l,j}
+ \sum_{l=k+1}^{m}(\bT_k)_{k,l}\bS_{l,j}
=
\sum_{l=1}^{k-1}\gamma\wT_{l,j}
+ \sum_{l=k+1}^{m}\gamma\bS_{l,j}$.

\emph{Subcase~2.1:} $j\leq k-1$.
Then
$(\bT_k\bS)_{i,j} = (\bT_k\bS)_{k,j}=\sum_{l=1}^{k-1}\gamma\wT_{l,j}
=\gamma\sum_{l=1}^{j-1}\wT_{l,j} + \gamma\sum_{l=j}^{k-1}\wT_{l,j}=
\gamma\sum_{l=1}^{j-1}\gamma (1+\gamma)^{l-1}+
\gamma\sum_{l=j}^{k-1}\big(\gamma (1+\gamma)^{l-1}-\gamma(1+\gamma)^{l-j}\big)
= \gamma^2\sum_{l=1}^{k-1}(1+\gamma)^{l-1} -
\gamma\sum_{l=j}^{k-1}\gamma (1+\gamma)^{l-j}
= \gamma(1+\gamma)^{k-1}-\gamma (1+\gamma)^{k-j}= \wT_{k,j} = \wT_{i,j}$.

\emph{Subcase~2.2:} $j=k$.
Then
$(\bT_k\bS)_{i,j} = (\bT_k\bS)_{k,k}=\sum_{l=1}^{k-1}\gamma\wT_{l,k}
=\gamma\sum_{l=1}^{k-1}\gamma (1+\gamma)^{l-1}=\gamma (1+\gamma)^{k-1}-\gamma
=\wT_{k,k}=\wT_{i,j}$.

\emph{Subcase~2.3:} $k+1\leq j$.
Then $(\bT_k\bS)_{i,j} = (\bT_k\bS)_{k,j}=
\sum_{l=1}^{k-1}\gamma\wT_{l,j} + \gamma
= \gamma\sum_{l=1}^{k-1}\gamma (1+\gamma)^{l-1} + \gamma
= \gamma(1+\gamma)^{k-1}=\wT_{k,j}=\wT_{i,j}$.

\emph{Case~3:} $i\geq k+1$.\\
Then $(\bT_k\bS)_{i,j} = \sum_{l=1}^{m}(\bT_k)_{i,l}\bS_{l,j} =
(\bT_k)_{i,i}\bS_{i,j}= \bS_{i,j}$.

\emph{Subcase~3.1:} $j=i$.
Then $(\bT_k\bS)_{i,j} = \bS_{i,i} = 1$.

\emph{Subcase~3.2:} $j\neq i$.
Then $(\bT_k\bS)_{i,j} = \bS_{i,j} = 0$.

Altogether, we have shown that
\begin{equation}
\bT_k\bS =
\begin{pmatrix}
\text{first $k$ rows of $\wT$}\\[+2mm]
0_{(m-k)\times k} \quad I_{m-k}
\end{pmatrix}.
\end{equation}
The ``In particular'' part is the case when $k=m$.
\end{proof}

\begin{corollary}
Let $k\in\{1,\ldots,m\}$. Then 
\begin{equation}
\bA^k =
\begin{pmatrix}
0_{(m-k)\times k} \quad I_{m-k}\\[+2mm]
\text{first $k$ rows of $\bT$}
\end{pmatrix};
\end{equation}
or entrywise
\begin{equation}
\big(\bA^k\big)_{i,j} =
\begin{cases}
1, &\text{if $1\leq i\leq m-k$ and $j=i+k$;}\\
0, &\text{if $1\leq i\leq m-k$ and $j\neq i+k$;}\\
\bT_{i+k-m,j}, &\text{if $m-k+1\leq i\leq m$.}
\end{cases}
\end{equation}
\end{corollary}
\begin{proof}
Proposition~\ref{p:tkproduct} yields
$\bA^{k}=L^{k}(\bT_{k}\bT_{k-1}\cdots \bT_{1})$ and the result now follows
from Proposition~\ref{l:1}. 
\end{proof}

\begin{remark}
We do not know whether it is possible to obtain a closed form for the powers of
$\bT$ that does not rely on the eigenvalue analysis from
Section~\ref{s:facts}. If such a formula exists, one may be able to
construct a different proof of Theorem~\ref{t:tfix} in this setting.
\end{remark}



\section{Kolmogorov means}\label{s:fmeans}

We now turn to moving Kolmogorov means, which are sometimes also
called $f$-means.

\begin{theorem}[moving Kolmogorov means]
\label{t:kolmo}
Let $D$ be a nonempty topological Hausdorff space, 
let $Z$ be a real Banach space, 
and let $f\colon D\to Z$ be an injective mapping such that
$\ran f$ is convex. 
Consider the linear recurrence relation 
\begin{subequations}
\begin{equation}
(\forall n\geq m)\quad 
y_{n}=f^{-1}\big(\alpha_{m-1}f(y_{n-1})+\cdots+ \alpha_{0}f(y_{n-m})\big),
\end{equation}
where
\begin{equation}(y_{0},\ldots, y_{m-1})\in D^m.
\end{equation}
\end{subequations}
Then the following hold: 
\begin{enumerate}
\item 
\label{t:kolmo1}
The sequence $(y_n)_\nnn$ is well defined and
the sequence $(f(y_n))_\nnn$ lies in the compact convex set
$\conv\{f(y_0),\ldots,f(y_{m-1})\}$. 
\item
\label{t:kolmo2}
Suppose that the {\rm\bf Basic~Hypothesis~\ref{bahy}} holds and
that $f^{-1}$ is continuous from $f(D)$ to $D$. 
Then
\begin{equation}
\label{e:kolmo2}
\lim_\nnn y_n = f^{-1}\bigg(\sum_{k=0}^{m-1}\lambda_kf(y_k)\bigg) = 
f^{-1}\bigg(
\frac{\sum_{k=0}^{m-1}a_kf(y_k)}{\sum_{k=0}^{m-1}a_k }\bigg)
=f^{-1}\bigg(\frac{\sum_{k=0}^{m-1}\sum_{i=0}^{k}\alpha_i
f(y_k)}{\sum_{k=0}^{m-1}\sum_{i=0}^{k}\alpha_i}\bigg). 
\end{equation}
\end{enumerate}
\end{theorem}
\begin{proof}
\ref{t:kolmo1}: This is similar to Remark~\ref{r:0614b} and proved
inductively.

\ref{t:kolmo2}: 
By Corollary~\ref{c:generalcase},
$(f(y_n))_\nnn$ converges to $\ell(f(y_0),\ldots,f(y_{m-1}))$,
which belongs to $\conv\{f(y_0),\ldots,f(y_{m-1})\}\subseteq
\conv f(D) = f(D) = \ran f$ by assumption and \ref{t:kolmo1}. 
Since $f^{-1}$ is continuous, the result follows from \eqref{e:0614a}. 
\end{proof}

Theorem~\ref{t:kolmo}\ref{t:kolmo2} allows for a universe of examples
by appropriately choosing $f$ and $(\alpha_0,\ldots,\alpha_{m-1})$. 
Let us present some classical moving means on (subsets) of the real
line with the equal weights $\alpha_0=\cdots=\alpha_{m-1}=1/m$.
Note that the {\rm\bf Basic~Hypothesis~\ref{bahy}} is satisfied. 

\begin{corollary}[some classical means]
Let $D$ be a nonempty subset of $\RR$, 
let $(y_0,\ldots,y_{m-1})\in D^m$. Here is a list of choices for $f$
in Theorem~\ref{t:kolmo}, along with the limits obtained by
\eqref{e:kolmo2}: 
\begin{enumerate}
\item\label{i:israel1}
If $D=\RR$ and $f=\Id$, then the moving 
\emph{arithmetic mean} 
sequence satisfies
\begin{equation}
y_{n}=\frac{1}{m}y_{n-1}+\cdots+\frac{1}{m}y_{n-m} \to
\frac{2}{m(m+1)}\sum_{j=0}^{m-1}(j+1)y_j.
\end{equation}
\item\label{i:israel2} 
If $D=\RR_{++}$ and $f=\ln$, then the moving 
\emph{geometric mean} 
sequence satisfies
\begin{equation}
y_{n+m}=\sqrt[m]{y_{n+m-1}\cdots y_{n}} \to
\bigg(\prod_{j=0}^{m-1}y_{j}^{j+1}\bigg)^{\frac{2}{m(m+1)}}. 
\end{equation}
\item\label{i:israel3} 
If $D=\RR_+$ and $f\colon x\mapsto x^p$, then the moving
\emph{H\"older mean} sequence satisfies
\begin{equation}
y_{n+m}=\bigg(\frac{1}{m}y_{n+m-1}^{p}+\cdots+\frac{1}{m}y_{n}^{p}\bigg)^{1/p}
\to
\bigg(\frac{2}{m(m+1)}\sum_{j=0}^{m-1}(j+1)y_{j}^{p}\bigg)^{1/p}.
\end{equation}
\item\label{i:israel4} 
If $D=\RR_{++}$ and $f\colon x\mapsto 1/x$, then the moving
\emph{harmonic mean} sequence satisfies
\begin{equation}
y_{n+m}=\bigg(\frac{1}{m}\frac{1}{y_{n+m-1}}+\cdots
+\frac{1}{m}\frac{1}{y_{n}}\bigg)^{-1} \to
\bigg(\frac{2}{m(m+1)}\sum_{j=0}^{m-1}(j+1)\frac{1}{y_{j}}\bigg)^{-1}.
\end{equation}
\end{enumerate}
\end{corollary}

Let us provide some means situated in a space of matrices.

\begin{corollary}[some matrix means]
Let $(Y_0,\ldots,Y_{m-1})\in\smpp$.
Then the following hold:
\begin{enumerate}
\item\label{i:israelm1}
The moving arithmetic mean satisfies
\begin{equation}
Y_{n}=\frac{1}{m}Y_{n-1}+\cdots+\frac{1}{m}Y_{n-m} 
\to \frac{2}{m(m+1)}\sum_{j=0}^{m-1}(j+1)Y_j.
\end{equation}
\item\label{i:israelm2} 
The moving harmonic mean satisfies
\begin{equation}Y_{n}=\bigg(\frac{1}{m}Y_{n-1}^{-1}+\cdots
+\frac{1}{m}Y_{n-m}^{-1}\bigg)^{-1} 
\to \bigg(\frac{2}{m(m+1)}\sum_{j=0}^{m-1}(j+1)Y_{j}^{-1}\bigg)^{-1}.
\end{equation}
\item\label{i:israelm3} 
The moving resolvent mean (see also {\rm \cite{bauschke2010}}\,)
satisfies
\begin{multline}
Y_{n}=\bigg(\frac{1}{m}(Y_{n-1}+\Id)^{-1}+\cdots+\frac{1}{m}(Y_{n-m}+\Id)^{-1}\bigg)^{-1}-\Id\\
\to
\bigg(\frac{2}{m(m+1)}\sum_{j=0}^{m-1}(j+1)(Y_{j}+\Id)^{-1}\bigg)^{-1}-\Id.
\end{multline}
\end{enumerate}
\end{corollary}
\begin{proof}
This follows from Theorem~\ref{t:kolmo} with $D=\smpp$ and 
$f=\Id$, $f=X\mapsto X^{-1}$, and
$f=X\mapsto (X+\Id)^{-1}$, respectively.
Note that matrix inversion is continuous; see, e.g., 
\cite[Example~6.2.7]{Meyer}. 
\end{proof}

\section{Moving proximal and epi averages for functions}
\label{s:lastsection}

In this last section, 
we apply the moving average (in particular Corollary~\ref{c:positivecoe}) 
to functions in the context of 
proximal and epi averages. Let us start by reviewing a key notion
(see also \cite{attouch,BorVanBook,Rock98}). 

\begin{definition}[epi-convergence]
{\rm (See \cite[Proposition 7.2]{Rock98}.)}
Let $g$ and $(g_n)_\nnn$ be functions from $X$ to $\RX$. 
Then
\begin{enumerate}
\item 
$(g_n)_{n\in\NN}$ \emph{epi-converges} to $g$, in symbols $g_n\aw g$,
if for every $x\in X$ one has
\begin{enumerate}
\item $\big(\forall\,(x_n)_{\nnn}\big)$ $x_n\to x \Rightarrow
g(x)\leq\varliminf g_n(x_n)$, and 
\item $\big(\exi (y_n)_\nnn\big)$ $y_n\to x$
and $\varlimsup g_n(y_n) \leq g(x)$;
\end{enumerate}
\item 
$(g_n)_{n\in\NN}$ \emph{pointwise converges} to $g$, in symbols $g_n\pw g$,
if for every $x\in X$ we have
$g(x)=\lim_{\nnn}g_{n}(x)$.
\end{enumerate}
\end{definition}

Let $g:X\rightarrow \RX$ be an extended real valued function.
Recall that
\begin{equation}
\big(g\Box\jj\big) \colon x\mapsto 
\inf_{y\in X}\Big(g(y)+\tfrac{1}{2}\|x-y\|^2\Big)
\end{equation}
is the \emph{Moreau envelope} of $g$. 
By \cite[Theorem~2.26]{Rock98}, if $g\in\GX(X)$, then 
$g\Box\jj$ is
convex and continuously differentiable on $X$.
Especially important are the following connections among epi-convergence 
pointwise convergence of envelopes, and the epi-continuity of Fenchel
conjugation.

\begin{fact}
{\rm (See \cite[Theorems 7.37 and 11.34]{Rock98}.)}
\label{f:berenstein}
Let $(g_n)_\nnn$ and $g$ be in $\GX(X)$.
Then
\begin{equation}
g_{n}\aw g 
\quad \Leftrightarrow \quad 
g_{n}\Box\jj \pw g\Box\jj
\quad \Leftrightarrow \quad 
g_{n}^*\aw  g^*.
\end{equation}
\end{fact}

We now turn to the moving proximal average (see \cite{bglw08} for
further information on this operation).

\begin{theorem}[moving proximal average]
\label{t:berenstein}
Let $(g_0,\ldots,g_{m-1})\in(\Gamma(X))^m$ and define
\begin{equation}
(\forall n\geq m)\quad
g_{n} =\big(\alpha_{m-1}(g_{n-1}+\jj)^* + \cdots +
\alpha_0(g_{n-m}+\jj)^*\big)^* -\jj.\label{e:p:otherforms:2}
\end{equation}
Then the following hold:
\begin{enumerate}
\item
\label{t:berenstein1}
The sequence $(g_n)_\nnn$ lies in $\Gamma(X)$.
\item
\label{t:berenstein2}
$(\forall n\geq m)$ 
$g_{n}\Box\jj=\alpha_{m-1}(g_{n-1}\Box\jj)+\cdots+ 
\alpha_{0}(g_{n-m}\Box\jj)$.
\item
\label{t:berenstein3}
Suppose that the 
{\rm\bf Basic~Hypothesis~\ref{bahy}} holds.
Then 
\begin{subequations}
\begin{equation}
g_n \aw
g=\big(\lambda_{m-1}(g_{m-1}+\jj)^* + \cdots +
\lambda_0(g_{0}+\jj)^*\big)^* -\jj.
\end{equation}
and 
\begin{equation}
g_n^* \aw
g^*=\big(\lambda_{m-1}(g^*_{m-1}+\jj)^* + \cdots +
\lambda_0(g^*_{0}+\jj)^*\big)^* -\jj.
\end{equation}
\end{subequations}
\end{enumerate}
\end{theorem}
\begin{proof}
\ref{t:berenstein1}: Combine 
\cite[Propositions~4.3 and 5.2]{bglw08}. 

\ref{t:berenstein2}: See \cite[Theorem~6.2]{bglw08}. 

\ref{t:berenstein3}:
It follows from \ref{t:berenstein2} and Corollary~\ref{c:positivecoe}
that 
$g_n\Box\jj\pw\sum_{k=0}^{m-1}\lambda_k(g_k\Box\jj)=g\Box\jj$.
The result now follows from Fact~\ref{f:berenstein} and
\cite[Theorem~5.1]{bglw08}. 
\end{proof}

We conclude the paper with the moving epi-average.
Recall that for $\alpha > 0$ and $g\in\Gamma(X)$,
$\alpha\timess g = \alpha g\circ \alpha^{-1}\Id$ and that
$0\timess g = \iota_{\{0\}}$.

\begin{theorem}[moving epi-average]
Let $(g_0,\ldots,g_{m-1})\in(\Gamma(X))^m$ and define
\begin{equation}
\label{e:0614z}
(\forall n\geq m)\quad
g_{n} = (\alpha_{m-1}\timess g_{n-1}) \Box \cdots \Box
(\alpha_0\timess g_{n-m}).
\end{equation}
Suppose that for every $k\in\{0,\ldots,m-1\}$, $g_k$ is cofinite and
that the {\rm\bf Basic~Hypothesis~\ref{bahy}} holds.
Then 
\begin{equation}
g_n\aw  g = (\lambda_{m-1}\timess g_{n-1}) \Box \cdots \Box
(\lambda_0\timess g_{n-m}).
\end{equation}
\end{theorem}
\begin{proof}
It follows from \cite[Proposition~15.7(iv)]{BC2011} that
$(\forall\nnn)$ $g_n\in\Gamma(X)$,
$g_n$ is cofinite, and $g_n^*$ is continuous everywhere. 
Furthermore, taking the Fenchel conjugate of \eqref{e:0614z} yields
\begin{equation}
(\forall n\geq m)\quad
g_{n}^* = \alpha_{m-1}g_{n-1}^* + \cdots +
\alpha_0 g_{n-m}^*.
\end{equation}
This and Corollary~\ref{c:positivecoe} imply 
not only that $g_n^*\pw g^*$ but also that the 
sequence $(g_n^*)_\nnn$ is equi-lsc everywhere 
(see \cite[pages~248f]{Rock98}). 
In turn, \cite[Theorem~7.10]{Rock98} implies that $g_n^*\aw g^*$.
The conclusion therefore follows from Fact~\ref{f:berenstein}. 
\end{proof}

\section*{Acknowledgments}
The authors are grateful to Jon Borwein and
Tam{\'a}s Erd{\'e}lyi for 
some helpful comments.
Heinz Bauschke was partially supported by the Natural Sciences and
Engineering Research Council of Canada and by the Canada Research Chair
Program.
Joshua Sarada was partially supported
by the Irving K.\ Barber Endowment Fund.
Xianfu Wang was partially
supported by the Natural Sciences and Engineering Research Council
of Canada.




\begin{thebibliography}{999}

\bibitem{AliBor} C.D.~Aliprantis and K.C.~Border, 
\emph{Infinite Dimensional Analysis}, 
third edition, Springer, 2006.

\bibitem{attouch} H. Attouch, \emph{Variational Convergence for Functions and Operators},
Applicable Mathematics Series, Pitman Advanced Publishing Program, Boston, MA, 1984.







%
\bibitem{BC2011}
H.H.\ Bauschke and P.L.\ Combettes,
\emph{Convex Analysis and Monotone Operator Theory in Hilbert Spaces},
Springer, 2011.


\bibitem{BWW}
H.H.\ Bauschke, X.\ Wang, and C.J.S.\ Wylie,
Fixed points of averages of resolvents: geometry and algorithms,
\texttt{http://arxiv.org/pdf/1102.1478v1}, February 2011.

\bibitem{BWWSIOPT}
H.H.\ Bauschke, X.\ Wang, and C.J.S.\ Wylie,
Fixed points of averages of resolvents: geometry and algorithms,
\emph{SIAM Journal on Optimization} 22 (2012), 24--40.

\bibitem{bglw08}
H.H. Bauschke, R. Goebel, Y. Lucet, and X. Wang, The proximal average: basic theory,
\emph{SIAM Journal on Optimization} 19 (2008), no. 2, 766--785.


\bibitem{bauschke2010}
H.H. Bauschke, S.M. Moffat and X. Wang, The resolvent average for positive semidefinite matrices,
\emph{Linear Algebra Appl.} 432 (2010), no. 7, 1757-1771.

\bibitem{berman}
A. Berman and R.J. Plemmons,
\emph{Nonnegative Matrices in the Mathematical Sciences},
SIAM, 1994.

\bibitem{BBS}
D.\ Borwein, J.M.\ Borwein, and B.\ Sims,
On the solution of linear mean recurrences,
preprint, June 2012,
\url{http://carma.newcastle.edu.au/jon/linearmeans.pdf}

\bibitem{BorBor}
J.M.\ Borwein and P.B.\ Borwein,
\emph{Pi and the AGM},
Wiley, New York, 1987.

\bibitem{BorVanBook}
J.M.\ Borwein and J.D.\ Vanderwerff,
\emph{Convex Functions},
Cambridge University Press, 2010.

\bibitem{BorErd}
P.\ Borwein and T.\ Erd{\'e}lyi,
\emph{Polynomials and polynomial inequalites},
Springer-Verlag, 1995. 


\bibitem{BoxJen} 
G.\ Box, G.\ Jenkins, and G.\ Reinsel, 
\emph{Time Series Analysis: Forecasting and Control}, third edition,
Prentice Hall, Englewood Cliffs, NJ, 1994.





\bibitem{chou} W.S. Chou, B.S. Du, and Peter J.-S. Shiue, A note on circulant transition matrices in Markov chains,
\emph{Linear Algebra and its Applications} 429 (2008), 1699--1704.

\bibitem{combettes02}
P.L. Combettes and T.  Pennanen,  Generalized Mann iterates for constructing fixed points in Hilbert spaces,
\emph{Journal of Mathematical Analysis and Applications} 275 (2002), no. 2, 521--536.



\bibitem{Davis}
P.J.\ Davis,
\emph{Circulant Matrices},
Wiley, 1979. 


%









%

%




\bibitem{Horn}
R.A. Horn and C.R. Johnson, \emph{Matrix Analysis}, Cambridge University Press, Cambridge, 1985.

\bibitem{Krafft}
O. Krafft and M. Schaefer, Convergence of the powers of a circulant stochastic matrix,
\emph{Linear Algebra and its Applications}~127 (1990), 59--69.


\bibitem{Meyer}
C.D.\ Meyer,
\emph{Matrix Analysis and Applied Linear Algebra},
SIAM, 2000.





\bibitem{Ortega}
J.M. Ortega,
\emph{Numerical Analysis:} A Second Course,
second edition, SIAM, 1990.

\bibitem{Ostrowski}
A.M.\ Ostrowski,
\emph{Solution of Equations and Systems of Equations},
second edition, Academic Press, New York and London, 1966.



\bibitem{Prasolov}
V.V.\ Prasolov,
\emph{Polynomials},
Springer-Verlag, Berlin, 2004.


\bibitem{Rock70}
R.T.\ Rockafellar,
\emph{Convex Analysis},
Princeton University Press, Princeton, 1970.


\bibitem{Rock98}
R.T.\ Rockafellar and R.J-B\ Wets,
\emph{Variational Analysis},
Springer, 
corrected third printing, 2009.

\bibitem{Tsokos} C.P. Tsokos, K-th moving, weighted and exponential moving average for time series forecasting models,
\emph{European Journal of Pure and Applied Mathematics} 3 (2010), 406--416.














\end{thebibliography}
\end{document}